\documentclass[a4paper,leqno,12pt]{amsart}
\usepackage{amsmath,amsthm}
\usepackage{amssymb}
\usepackage[all]{xypic}
\usepackage[all,cmtip]{xy}
\usepackage{tikz}
\usepackage{tikz-cd}
\usetikzlibrary{matrix,shapes,arrows,positioning,chains}
\usepackage{url}
\usepackage{hyperref}

\SelectTips{cm}{12}
\UseTips{}
\usepackage{euscript}
\begin{document}

\theoremstyle{plain}
\swapnumbers
 \newtheorem{theorem}{Theorem}[section]
\newtheorem{lemma}[theorem]{Lemma}
\newtheorem{corollary}[theorem]{Corollary}
\newtheorem{proposition}[theorem]{Proposition}

\newtheorem*{thma}{Theorem A}
\newtheorem*{thmb}{Theorem B}
\newtheorem*{thmc}{Theorem C}
\newtheorem*{thmd}{Theorem D}
\newtheorem*{thme}{Theorem E}

\theoremstyle{definition}
\newtheorem{definition}[theorem]{Definition}
\newtheorem{example}[theorem]{Example}
\newtheorem{prop}[theorem]{Proposition}
\newtheorem{problem}[theorem]{Problem Statement}
\newcommand{\R}{\mathbb{R}}
\newcommand{\Z}{\mathbb{Z}}
\theoremstyle{remark}
\newtheorem{conv}[theorem]{Convention}
\newtheorem{fact}[theorem]{Fact}
\newtheorem{ob}[theorem]{Observation}
\newtheorem{remark}[theorem]{Remark}
\newtheorem{ack}[theorem]{Acknowledgement}

%
%
\title[Flip Stiefel Manifold]{A study of topology of the flip Stiefel manifolds}
\author[S.~Basu]{Samik Basu}
\address{Stat-Math Unit,
Indian Statistical Institute,
B. T. Road, Kolkata-700108, India.}
\email{samik.basu2@gmail.com; samikbasu@isical.ac.in}
\author[S.~Gondhali]{Shilpa Gondhali}
\address{Department of Mathematics\\ Birla Institute of Technology and Science (BITS)-Pilani, K K Birla
Goa Campus\\ 403726 Goa, India.}
\email{shilpag@goa.bits-pilani.ac.in, shilpa.s.gondhali@gmail.com}
\author[F.~Safikaa]{ Fathima Safikaa}
\address{Department of Mathematics\\ Birla Institute of Technology and Science (BITS)-Pilani, K K Birla
Goa Campus\\ 403726 Goa, India.}
\email{p20180409@goa.bits-pilani.ac.in, fsafikaa@gmail.com }
\date{\today}
\subjclass{Primary : 57T15, 57S17  Secondary : 57R20, 55T10}
\keywords{Flip Stiefel manifolds, Serre Spectral Sequence, Equivariant maps, Parallelizability, Span, Stable span, Stiefel Whitney classes}

\begin{abstract}
A well known quotient of the real Stiefel manifold is the projective Stiefel manifold.  We introduce a new family of quotients of the real Stiefel manifold by cyclic group of order 2 whose action is induced by simultaneous pairwise flipping of the coordinates. We obtain a description for their tangent bundles, compute their mod 2 cohomology and compute the Stiefel Whitney classes of these manifolds. We use these to provide applications to their stable span, parallelizability and equivariant maps, and some associated results in topological combinatorics.
\end{abstract}
\maketitle

\newcommand{\mR}{{\mathbb R}}
\newcommand{\mC}{{\mathbb C}}
\newcommand{\diag}{\text{diag}}
\newcommand{\nf}[2]{N_F({#1},{#2})}
\newcommand{\np}[2]{N_P({#1},{#2})}
\newcommand{\mZ}{{\mathbb Z}}
\newcommand{\mS}{{\mathbb S}}
\newcommand{\mRP}{{\mathbb RP}}

\section{Introduction}\label{sintroduction}

We know that the Stiefel manifold was used by Adams \cite{a} in his path-breaking work on Vector fields on spheres. We begin by recalling its definition.
\begin{definition}
Let $1 \leq k \leq n$. The real Stiefel manifold, denoted by $V_{n,k}$, is the set of all $k$-tuples of linearly independent vectors in $\mR^n.$
\end{definition}

{{} Notice that $V_{n,k} \subseteq (\mR^n)^k$ and it inherits the subspace topology of $(\mR^n)^k.$ Moreover, }$V_{n,k}$ is homotopically equivalent (using {}{the} Gram Schmidt process) to {{} the subspace} of all ordered $k$-tuples of unit vectors in $\mR^n$ which are pairwise orthogonal with respect to the Euclidean inner product. {{} Using the Orbit-Stabilizer theorem, $V_{n,k}$ is identified with} $O(n)/O(n-k)$ and hence, {{}$V_{n,k}$ has a  quotient topology. It is a fact that the quotient topology and subspace topology on $V_{n,k}$ are equivalent \cite[Page 14,15]{j}.} Note that,  $V_{n,1}= \mS^{n-1}$ and $V_{n,n}= O(n).$  \\

The projective Stiefel manifold, which is a quotient of the Stiefel manifold, has been extensively studied in literature \cite{LS80,GH68}. 

\begin{definition}\label{d1.2}
The real projective Stiefel manifold, denoted by $PV_{n,k}$, is defined to be the quotient of $V_{n,k}$ modulo the free action of the abelian group of order 2, denoted by $C_2:= \{ \pm 1\}$ which acts via scalar multiplication:
\[
z\cdot(v_{1},\ldots, v_{k}) = (zv_{1}, \ldots, zv_{k})\text{ for}~  (v_1, \ldots, v_k) \in V_{n,k} ~\text{and}~ z \in C_2.
\]
\end{definition}

Observe that $PV_{n,k}$ is a homogeneous space. $PV_{n,k} = O(n) / (C_2 \times O(n-k))$, where $C_2$ is identified with the set \(\{ \pm I_{k \times k}\}.\)  In particular, $PV_{n, 1}= \mR P^{n-1},$ the real projective space.\\

Although the Stiefel manifold is intensively discussed in \cite{j}, not much has been explored about its quotients. We refer to \cite[Section 2]{s} for a detailed discussion. We consider a class of quotients of the real Stiefel manifold, called the {\it flip Stiefel  manifold} which is defined as follows.

\begin{definition}\label{d1.3}
Let $1 < 2k \leq n$. Consider an action of $C_2= \{e, a\},$ the group of order two, on $V_{n, 2k}$ given by
\[
a \cdot (v_1, v_2, \ldots, v_{2k})= (v_2, v_1, v_4, v_3, \ldots, v_{2k}, v_{2k- 1})
\]
where $(v_1, v_2, \ldots, v_{2k}) \in V_{n, 2k}.$ A {\it{flip Stiefel manifold}}, denoted by the symbol $FV_{n,2k}$, is the quotient space of $V_{n,2k}$ under the given action of $C_2.$
\end{definition}

Another way to describe this action is using matrix representation.
\begin{itemize}
\item Represent each element $v = (v_1, v_2, \ldots, v_{2k}) \in V_{n,2k}$ as an $n\times 2k$ matrix with entries $v_{1}, v_{2}, \ldots, v_{2k}$ as mutually perpendicular column vectors of $\mR^n$.
\item We see that the action of $a$ on $V= (v_1, v_2, \ldots, v_{2k})$ is given by multiplication from right by a permutation matrix $F$ of order $2k \times 2k.$
\item To be precise,
\[
a \cdot V=
\begin{matrix}
\begin{matrix} v_1 & v_2 &\cdots & v_{2k} \end{matrix} & \\
\begin{bmatrix}
v_{11} & v_{12} & \cdots & v_{12k} \\
v_{21} & v_{22} & \cdots & v_{22k}\\
\vdots & \ddots & \ddots & \vdots \\
v_{n1} & v_{n2} & \cdots & v_{n2k}
\end{bmatrix} &
\begin{matrix} \times  \end{matrix}&
\begin{matrix} F \end{matrix}

\end{matrix}\]
\end{itemize}
where $F= \diag(f, f,\ldots, f)$ with $f= \begin{bmatrix}
0 & 1\\
1 & 0
\end{bmatrix}.$\\

\noindent A straightforward argument gives:
\begin{lemma}\label{l1.4}
The flip Stiefel manifold \(FV_{n,2k}\) is a homogeneous space, and it can be identified with \(\frac{O(n)}{C_2 \times O(n-2k)} \), where \(C_2= \{ I_{2k\times 2k},F\}\) and \(F\) is the matrix defined as above.
\end{lemma}

\begin{proof}
This follows using the orbit-stabilizer theorem. Take \(G= O(n)\) and consider the action of \(G\) on \(FV_{n,2k}\) given by
\[
A \cdot [(v_1, v_2, \ldots, v_{2k})]= [(Av_1, Av_2, \ldots, Av_{2k})] ~\text{for all}~ A \in G ~\text{and}~  [(v_1, v_2, \ldots, v_{2k})] \in FV_{n,2k}.
\]
Since similar action of \(G\) on \(V_{n,2k}\) is transitive and the map \(V_{n,2k} \to FV_{n,2k}\) is a surjection, the action of \(G\) on \(FV_{n,2k}\) is transitive. Hence, we have
\[
FV_{n,2k}= G/ S_{v_0}
\]
where $S_{v_0}$ is the stabilizer of an element $v_0$ of $FV_{n,2k}$. Consider \(v_0= [(e_1, e_2, \ldots, e_{2k})]\) where \(e_1, e_2, \ldots, e_{2k}\) are the standard basis elements of \(\mR^n.\)\\
Now, we see that
\[
S_{v_0}= \big\{\diag(A, B)\mid A \in \{I_{2k\times 2k}, F\}, B \in O(n-2k)\big\} \cong C_2 \times O(n-2k).
\]
This completes our proof.
\end{proof}

\begin{remark}
(1) Notice that the topology we obtain using the identification of  \(FV_{n,2k}\) with  \(\frac{O(n)}{C_2 \times O(n-2k)} \) is equivalent to the quotient topology on \(FV_{n,2k}\) obtained using the quotient map \(V_{n,2k} \to FV_{n,2k}.\) \\
(2) The quotient map $V_{n,2k} \to FV_{n,2k}$ is a double cover and hence the dimension of the manifold $FV_{n,2k}$ is the same as the dimension of $V_{n,2k}.$ Hence,
\[
\text{dim}(FV_{n,2k})= \frac{n(n-1)}{2}- \frac{(n-2k)(n-2k-1)}{2}= k(2n-2k-1).
\]
\end{remark}

In this paper, we begin by describing the tangent bundle $TFV_{n,2k}$ in Section \ref{stangent}, further obtaining expressions for {}{the} (additive structure of) mod-2 cohomology of $FV_{n,2k}$ in Section \ref{scohomology}. We then use it to get information about non-parallelizability, span and stable span in Section \ref{sparallelizability}. Finally, we conclude by discussing some equivariant maps {{} and  related geometric consequences} in the last section, Section \ref{sequivariant}.  The following are a few results that we prove in the paper.

\begin{thma}[Parallelizability]
$FV_{n,2k}$ is not parallelizable when
\begin{enumerate}
\item $k \equiv 1 \pmod{2}$ and $n \equiv 1 \pmod{2},$
\item $k \equiv 1 \pmod{4}$ and $n \equiv 0 \pmod{4},$
\item $k \equiv 2 \pmod{4}$ and $n \equiv 0 \pmod{4},$
\item $k \equiv 2 \pmod{4}$ and $n \equiv 2 \pmod{4},$
\item $k \equiv 3 \pmod{4}$ and $n \equiv 2 \pmod{4}.$
\end{enumerate}
\end{thma}

\begin{thmb}[Span and stable span]
Span ($FV_{n,2k}$)= Stable span ($FV_{n,2k}$) when
\begin{enumerate}
\item $k \equiv 0 \pmod{4}$ or  $k \equiv 2 \pmod{4},$
\item $k \equiv 1 \pmod{4}$ and $n \equiv 0 \pmod{4}$ for $k>1,$
\item $k \equiv 1 \pmod{4}$ and $n \equiv 2 \pmod{4}$ for $k>1,$
\item $k \equiv 3 \pmod{4}$ and $n \equiv 0 \pmod{4}$ for $k>3.$
\end{enumerate}
\end{thmb}

\begin{thmc}[Equivariant maps]
Let $H'= \{ \pm 1\}$ and $H =\{ I_{2k\times 2k}, F\}$ where $F$ is as defined earlier. Let $(H', V_{n,2k})$ denote the action of $H'$ on $V_{n,2k}$ by antipodal identification as defined in  Definition \ref{d1.2} and $(H, V_{n,2k})$ denote the action of $H$ on $V_{n,2k}$ as defined in Definition \ref{d1.3}.
\begin{enumerate}
\item There doesn't exist any $H$-equivariant map $f:V_{n,2k} \rightarrow V_{m,2l}$, for $ k = 1+ \sum\limits_{i=0}^{r-1}2^i $ and $l = 1 + \sum\limits_{ i = 0 }^{s-1}a_i2^i$, $a_i = 0,1$ and $r > s$.
\item There doesn't exist any equivariant map $f:(H', S^{n-1}) \rightarrow (H,V_{n,2k})$, for $ n = 2^r -1$ and $ r >0$.
\item If there  exists any equivariant map $f: (H, O(2k)) \rightarrow (H, O(2l))$ then  $ k \geq l$ and $ k \equiv 1 \pmod{2} $ and $l \equiv 0 \pmod{2}$.
\end{enumerate}
\end{thmc}

\begin{thmd}\label{td}
\begin{enumerate}
\item For any function $f:S^{n-1} \to \mR^{n-2}$, there exist orthogonal elements $v_1, v_2$ of $S^{n-1}$ such that $f(v_1)=f(v_2)$.  
\item For any function $f:S^{n-1}\to \mR^{\lfloor{\frac{n}{2}}\rfloor-2}$, there exist mutually orthogonal elements $v_1,v_2, v_3, v_4$ of $S^{n-1}$ such that $f(v_1)=f(v_2)$ and $f(v_3)=f(v_4)$. 
\item For any function $f: S^{n-1} \to \mR ^{\lfloor\frac{n-r_n-1}{3}\rfloor}$, there exist mutually orthogonal elements 
$v_1,v_2,v_3,v_4,v_5, v_6$ of \(S^{n-1}\) such that $f(v_1)=f(v_2)$, $f(v_3)=f(v_4)$, $f(v_5)=f(v_6)$,  where 
\[r_n=\begin{cases} 5 &\mbox{ if } n \equiv 1, 2 \pmod{4},\\
 4 &\mbox{ if } n \equiv 0 \pmod{4}, \\
 3 &\mbox{ if } n \equiv 3 \pmod{4}. \end{cases} \] 
\end{enumerate}
\end{thmd}

\begin{thme}
\begin{enumerate}
\item Given $n-2$ convex compact regions in $\mR^n$, there is an orthogonal decomposition $\mR^n=\mR^2 \oplus \mR^{n-2}$ such that each region is inscribed in  $S \times \mR^{n-2}$ where $S$ is a  square. 
\item Given $\lfloor{\frac{n}{2}}\rfloor-2$ convex compact regions in $\mR^n$, there is an orthogonal decomposition $\mR^n=\mR^2 \oplus \mR^2 \oplus \mR^{n-4}$ such that each region is inscribed in $S \times S \times \mR^{n-4}$ where $S$ is a  square.
\item Given $\lfloor\frac{n-r_n-1}{3}\rfloor$ convex compact regions in $\mR^n$ (with $r_n$ as defined above), there is an orthogonal decomposition $\mR^n=\mR^2 \oplus \mR^2\oplus \mR^2 \oplus \mR^{n-6}$ such that each region is inscribed in $S \times S \times S \times \mR^{n-6}$ where $S$ is a  square. 
\end{enumerate}
\end{thme}

\section{Tangent Bundle}\label{stangent}

\noindent Our plan is to study the topology of the flip Stiefel manifold. We begin by examining its parallelizability. For computation of the tangent bundle $TFV_{n,2k},$ we use the real flag manifold which is defined as follows.

\begin{definition}
 Let $n_1, n_2, \ldots, n_s$ be positive integers such that $n_1+ n_2+ \cdots+ n_s= n.$ By an "$(n_1, n_2, \ldots, n_s)$-flag over $\mathbb{R}$", we mean a collection $\sigma$ of mutually orthogonal subspaces $(\sigma_1, \sigma_2, \ldots, \sigma_s)$ of $\mR^n$ such that  $\text{dim}_{\mathbb{R}}\sigma_i= n_i.$ The space of all such flags is a compact manifold. Such a manifold is denoted by $G_{\mathbb{R}}(n_1, n_2, \ldots, n_s).$ 
 \end{definition}
 
Note that $G_{\mathbb{R}}(n_1, n_2)$ is the Grassmanian manifold of $n_1$ planes in the Euclidean $n_1+ n_2$ space.
Using similar arguments as earlier, we can identify $G_{\mathbb{R}}(n_1, n_2, \ldots, n_s)$ with $O(n)/\big(O(n_1)\times O(n_2)\times \cdots \times O(n_s)\big).$\\

\noindent We use information about the tangent bundle of the flag manifold  to compute the tangent bundle of $FV_{n,2k}$. Consider the map $f \colon FV_{n,2k} \to G_\mathbb{R}(2,2,\ldots,2,n- 2k)$ given by
\[
f([(v_1, v_2, \ldots, v_{2k})])= (\mathbb{R}v_1 \oplus \mathbb{R}v_2, \ldots, \mathbb{R}v_{2k-1} \oplus \mathbb{R}v_{2k}, \langle v_1, v_2, \ldots, v_{2k}\rangle^{\perp}).
\]

The map $f$ is a fiber bundle with fiber $(\underbrace{O(2) \times \cdots \times O(2)}_{k~ \text{copies}})/C_2$. Hence, we have
\[
TFV_{n,2k}\cong f^\ast TG_\mathbb{R}(2, 2, \ldots, 2, n- 2k) \oplus \alpha,
\]
where $\alpha$ is the bundle along the fiber and \(f^\ast TG_\mathbb{R}(2,2,\ldots,2,n- 2k) \) denotes the pull back bundle of the tangent bundle on the flag manifold $G_\mathbb{R}(2, 2, \ldots, 2, n- 2k)$. 

\begin{conv}
We use the symbol $\cong$ to denote isomorphism as real vector bundles throughout the text. 
\end{conv}

\noindent We elaborate on a significant property of the fiber of \(f: FV_{n,2k} \to G_{\mR}(2,2,\ldots, n-2k)\) that aids in acquiring an explicit description of $\alpha.$

\begin{lemma}
The fiber $(O(2))^k/C_2$ is homeomorphic to the Lie group $SO(2) \times (O(2))^{k-1},$ where $C_2$ is as described in Lemma \ref{l1.4}. 
\end{lemma}
\begin{proof}
\noindent We describe the identification between the two sets here. 

An arbitrary element in this fiber consists of a set of diagonally aligned block matrices of $O(2)$ with the simultaneous action of $C_2$ on all of these matrices. Hence it is of the form $\diag(A_1, A_2, \cdots, A_k)$ where $A_i = \left\{\begin{pmatrix}
a_i & b_i  \\
c_i & d_i \\
\end{pmatrix}, \begin{pmatrix}
b_i & a_i  \\
d_i & c_i \\
\end{pmatrix} \right\}
$ for every $i$. By fixing the element with a positive determinant of the pair in the first block $A_1,$ we have a unique identification of $(O(2))^k/C_2$ with  $SO(2) \times (O(2))^{k-1}.$

Verifying that the identification described above is a homeomorphism is a routine task.
\end{proof}

\begin{corollary}
(1) The fiber bundle  $(O(2))^k/C_2 \rightarrow FV_{n,2k} \xrightarrow{f} G_\mathbb{R}(2,\ldots,2,n- 2k) $,  is a principal bundle.\\
(2) The bundle along the fiber of the map $f$ is trivial. That is, $\alpha \cong k \epsilon_{\mathbb{R}}$, where $\epsilon_{\mathbb{R}}$ denotes the trivial line bundle.
\end{corollary}
\begin{proof}
{{}
Part (1) of the Corollary is an immediate consequence of \cite[Page 50]{kn}. We briefly discuss proof for (2) here. Using \cite[page 35, Problem 3-A]{ms} we know that for a principal bundle $G \to E \xrightarrow{f} B$, we have
\[
\tau E \cong \alpha \oplus f^\ast \tau B
\]
where $\alpha,$ called the bundle along the fiber, is built up from the kernels of $Df(x): T_xE \to T_{f(x)}B.$ 
We have 
\[
exp: \mathfrak{g}= T_eG \to G
\]
Hence we get
\[
A: E \times T_eG \xrightarrow{(id, exp)} E \times G \xrightarrow{action} E.
\]
Taking derivative, we write 
\[ L : E \times T_e G \to TE, ~~L(x,g)= DA(x,0)(g).\]
Clearly the image of $L$ lands in $T_xE$, and as the map $E\times T_e G \xrightarrow{A} E \to B$ equals $f$, the image of $L$ also lands in the kernel of $Df$. It follows that $L$ is an isomorphism between $E\times T_e G $ and the Kernel of $Df$.   
}
\end{proof}

From the well-noted work of Lam (see \cite[Page 306]{l}), we know that
\[
TG_\mathbb{R}(n_1,n_2,\ldots,n_s) \cong \bigg(\underset{1\leq i< j\leq s}{\bigoplus}  \xi_i \otimes \xi_j \bigg)
\]
where $\xi_i$ is the real vector bundle of rank \(n_i,\) for \(1 \leq i \leq s,\) whose fiber at the point \(\sigma= (\sigma_1, \sigma_2, \ldots, \sigma_s)\) is the vector space \(\sigma_i.\) It is also noted that there is an isomorphism

\begin{equation}\label{esumxi}
\xi_1 \oplus \xi_2 \oplus \cdots \oplus \xi_s \cong n\epsilon_\mR.
\end{equation}

\noindent We see that
\[
TFV_{n,2k} \cong \bigg(\underset{1\leq i< j\leq k+1}{\bigoplus} f^\ast \xi_i \otimes f^\ast \xi_j \bigg) \oplus  k\epsilon_{\mR}.
\]
\noindent {{}Consider a (real) plane bundle \(\zeta_{n,2k}\) over $FV_{n,k},$ with the total space
\[
E[\zeta_{n,2k}]= \{ [v, (s,t)]: v \in V_{n,2k}; s, t\in \mR ~\text{with}~ (v, (s,t))\sim (\tilde{v}, (t,s))\},
\]
where \(\tilde{v}= (v_2, v_1, \ldots, v_{2k}, v_{2k-1})\) when \(v= (v_1, v_2, \ldots, v_{2k}).\) Then, by using the definition of the induced bundle, we see that \(f^\ast \xi_i\) can be identified with $\zeta_{n,2k}$ for $1 \leq i \leq k$. Let \(\gamma_{n,2k}:=f^\ast\xi_{k+1}.\)} Substituting this, we get

\begin{lemma}
The tangent bundle $TFV_{n,2k}$ is given  by
\begin{equation}\label{tgtflip}
TFV_{n,2k}\cong \frac{k(k-1)}{2}(\zeta_{n,2k}\otimes \zeta_{n,2k}) \oplus k(\zeta_{n,2k}\otimes \gamma_{n,2k}) \oplus  k\epsilon_{\mR}.
\end{equation}
\end{lemma}
\begin{proof}
\begin{align*}
TFV_{n,2k} &\cong \bigg(\underset{1\leq i< j\leq k}{\bigoplus} \zeta_{n,2k}\otimes \zeta_{n,2k}\bigg) \oplus \bigg(\bigoplus\limits_{i=1}^{k}\zeta_{n,2k}\otimes \gamma_{n,2k}\bigg) \oplus  k\epsilon_{\mR}\\
&\cong \bigg(\underset{1\leq i< j\leq k}\sum 1\bigg) (\zeta_{n,2k}\otimes \zeta_{n,2k}) \oplus k(\zeta_{n,2k}\otimes \gamma_{n,2k}) \oplus  k\epsilon_{\mR}\\
&\cong \bigg(\sum\limits_{j=2}^k (j-1)\bigg) (\zeta_{n,2k}\otimes \zeta_{n,2k}) \oplus k(\zeta_{n,2k}\otimes \gamma_{n,2k}) \oplus  k\epsilon_{\mR}\\
&\cong \bigg(\sum\limits_{J=1}^{k-1} J\bigg) (\zeta_{n,2k}\otimes \zeta_{n,2k}) \oplus k(\zeta_{n,2k}\otimes \gamma_{n,2k}) \oplus  k\epsilon_{\mR}
\end{align*}

which proves Equation (\ref{tgtflip}).
\end{proof}

\begin{lemma}
The real plane bundle $\zeta_{n,2k}$ is stably isomorphic to the line bundle $\xi_{n,2k}.$ That is,
\[
\zeta_{n,2k} \cong \epsilon_\mR \oplus \xi_{n,2k}
\]
where \(\xi_{n,2k}\) is the associated real line bundle to the double cover \( V_{n,2k} \to FV_{n,2k}\). 
\end{lemma}

\begin{proof}
We make use of the ring homomorphism \(\Phi: R_\mR(C_2) \to K_\mR^\ast(FV_{n,2k})\) where  \( R_\mR(C_2)\) denotes the representation ring of (equivalent) real representations of \(C_2\) and \(K_\mR^\ast(FV_{n,2k})\) denotes the real \(K\) ring of \(FV_{n,2k}\) given by
\[
\rho \colon C_2 \to GL(V) \mapsto \xi_\rho,~\text{where}~ E[\xi_\rho]:= V_{n,2k} \underset{\rho}{\times} V.
\]
(See \cite[Chapter 13, Section 5]{h} for details.) Notice that the image of regular representation is \(\zeta_{n,2k}\) and the image of  direct sum of the trivial and sign representation is \(\epsilon_\mR \oplus \xi_{n,2k}\). 

Using the irreducible decomposition of representations of $C_2,$ {}{(refer \cite[Part 1]{f} for a detailed discussion on representations of finite groups)} we see that the regular representation is isomorphic to the direct sum of the trivial and sign representations of $C_2.$

Now, as the regular representation and the representation obtained by the direct sum of the trivial and sign representation are equivalent representations, images under \(\Phi\) are isomorphic vector bundles. 
\end{proof}

\noindent We end this section with the following lemma which we use later to get information about parallelizability, span and stable span.

\begin{lemma}
The tangent bundle $TFV_{n,2k}$ satisfies a relation given by
\begin{equation}\label{stableiso}
\frac{k(k+1)}{2}(\zeta_{n,2k} \otimes  \zeta_{n,2k})  \oplus TFV_{n,2k} \cong nk\zeta_{n,2k} \oplus k\epsilon_{\mR}.
\end{equation}
\end{lemma}
\begin{proof}
From Equation (\ref{esumxi}), we have that
\[
\xi_1 \oplus \xi_2 \oplus \cdots \oplus \xi_s \cong n\epsilon_\mR.
\]
Applying pull backs on both sides, we get
\[
\bigg(\bigoplus_{i=1}^k f^\ast \xi_i \bigg) \oplus f^\ast\xi_{k+1} \cong n\epsilon_\mR.
\]
Using notations defined earlier, we have
\[
\bigg( \bigoplus_{i=1}^k \zeta_{n, 2k} \bigg) \oplus \gamma_{n, 2k} \cong n\epsilon_\mR.~\text{That is,}~ k\zeta_{n, 2k} \oplus \gamma_{n, 2k} \cong n\epsilon_\mR.
\]
Multiplying by $k$ and tensoring with $\zeta_{n,2k}$ on both sides,
\[
k^2(\zeta_{n,2k} \otimes \zeta_{n,2k})  \oplus k(\zeta_{n,2k} \otimes  \gamma_{n, 2k}) \cong nk\zeta_{n,2k}.
\]
Comparing the above with Equation (\ref{tgtflip}), we rewrite it as
\[
\frac{k(k+1)}{2}(\zeta_{n,2k} \otimes  \zeta_{n,2k})  \oplus  \frac{k(k- 1)}{2}(\zeta_{n,2k} \otimes \zeta_{n,2k})  \oplus k(\zeta_{n,2k} \otimes  \gamma_{n, 2k}) \oplus k\epsilon_{\mR} \cong nk\zeta_{n,2k} \oplus k\epsilon_{\mR}.
\]
which is nothing but Equation \((\ref{stableiso})\) and hence the proof is complete.

\end{proof}

For analysing parallelizability, we need information about the cohomology of $FV_{n,2k}.$\\

\section{Cohomology of  \(FV_{n,2k}\)}\label{scohomology}

\noindent Analogical to \cite[Page 236-237]{ss}, we construct a sequence of commutative fibrations, and using Serre spectral sequences, we compute the cohomology of $FV_{n,2k}$  with $\mZ/2\mZ$ coefficients.\\

\noindent From the principal $C_2$-fibration $V_{n,2k} \xrightarrow{\pi} FV_{n,2k}$, we obtain the fibration 
\[
FV_{n,2k} \xrightarrow{\Pi} BC_2 = \mRP^{\infty},
\]
where $\Pi$ is homotopic to the map $EC_2 \underset{C_2}{\times} V_{n,2k} \to BC_2$. 

\noindent {{} Now, we recall a fact that follows from \cite[Page 501]{m}.}
\begin{fact}
Let $X$ be a $G$-space, where $G$ is a discrete group, with group action denoted by $\Phi.$ In the fibration
\[
\xymatrix{
X \ar[r] & EG\underset{G}{\times} X \ar[r] & BG,
}
\]
the action of the fundamental group $\pi_1(BG)= G$ on the fiber $X$ is the same as the $G$-action $\Phi.$
\end{fact}

From the following proposition, we infer that the fibration $V_{n,2k} \to FV_{n,2k} \xrightarrow{\Pi} BC_2 = \mRP^{\infty}$ satisfies the hypotheses required of the Leray-Serre Spectral Sequence Theorem(See \cite[Page 138, 163]{m}).

\begin{proposition}
$\pi_1(\mRP^{\infty})$ acts trivially on $H^\ast(V_{n,2k}; \mZ/2\mZ).$
\end{proposition}
\begin{proof}
The action of $\pi_1(\mRP^{\infty})= C_2$ on $V_{n,2k}$ is defined to be the right multiplication of its elements by the diagonal matrix $ F= diag(f, \cdots, f), $ where $f = \begin{pmatrix}
0 & 1  \\
1 & 0\\
\end{pmatrix}$.

 When k is even, $\det(F)$ = 1, and therefore, $F$ lies in the identity component of $O(2k)$. Since $SO(2k)$ is path-connected, the given action is homotopic to the trivial action on $V_{n,2k}.$ This implies that $\pi_1(\mRP^{\infty})$ acts trivially on $H^*(V_{n,2k}).$
 
When k is odd, $\det(F) = -1$. However, for $n \geq 2k+1,$ the projection $V_{n,2k+2} \xrightarrow{pr} V_{n,2k}$ is $C_2$-equivariant. Also, $H^\ast(V_{n,2k}; \mZ/2\mZ) \xrightarrow{pr^\ast} H^\ast(V_{n,2k+2}; \mZ/2\mZ)$ is injective and  $C_2$-equivariant. Since the action of  $C_2$  is trivial on $V_{n,2k+2}$, equivariance forces  $C_2$ to act trivially on  $H^\ast(V_{n,2k};\mZ/2\mZ).$
\end{proof}

We also have the principal $(O(2))^k$ fibration $ V_{n,2k} \xrightarrow{\theta}  G_{\mR}(2, \ldots, 2, n-2k)$ that gives us the fibration $ Flag_{\mR}(2, \ldots, 2, n-2k) \xrightarrow{\Theta} B((O(2))^k) = (G_2(\mR^{\infty}))^k$, where $G_2(\mR^{\infty})$ is the Grassmanian of planes in $\mR^{\infty}$  and  $\Theta$ is homotopic to the map 
\[
E((O(2))^k) \underset{O(2)^k}{\times} V_{n,2k} \rightarrow B((O(2))^k).
\]
Also, the Stiefel manifold fibers over the Grassmanian $ G_{\mR}(2k, n-2k)$  giving us a principal $O(2k)$ fibration $V_{n,2k} \xrightarrow{\psi}   G_{\mR}(2k, n-2k)$, which gives us a fibration on the classifying space of $O(2k)$, 
\[
G_{\mR}(2k, n-2k) \xrightarrow{\Psi}B(O(2k)) = G_{2k}(\mR^{\infty}).
\]
Now the  $C_2$ action on $(O(2))^k$ and the canonical inclusion of $(O(2))^k$ into $O(2k)$ induces corresponding  actions on their classifying spaces  and hence we have the following commutative fibrations,

\[\begin{tikzcd}[row sep=2em,column sep=5em]
 C_2 \arrow[r, "g"] \arrow[d,  ] & (O(2))^k  \arrow[r, "h"] \arrow[d] & O(2k)  \arrow[d] \\
V_{n,2k}  \arrow[r, equal] \arrow[d,"\pi" ]&V_{n,2k} \arrow[r, equal] \arrow[d, "\theta"]  &V_{n,2k}  \arrow[d, "\psi"]\\
FV_{n,2k}  \arrow[r] \arrow[d, "\Pi"]&  G_{\mR}(2, \cdots, 2, n-2k)  \arrow[r] \arrow[d, "\Theta"] & G_{2k}(\mR^n)  \arrow[d, "\Psi"] \\
\mRP^{\infty}  \arrow[r, "\hat{g}"]  &  (G_2(\mR^{\infty}))^k  \arrow[r, "\hat{h}"]  & G_{2k}(\mR^{\infty})
\end{tikzcd}\]

These fibrations induce Serre spectral sequences,
\[ E_2^{p,q} = H^p( G_{2k}(\mR^{\infty}) \otimes H^q(V_{n,2k}) \Rightarrow H^{p+q}(G_{2k}(\mR^n),\]
\[ E_2^{p,q} = H^p(\mRP^{\infty} ) \otimes H^q(V_{n,2k}) \Rightarrow H^{p+q}(FV_{n,2k}).\]

The composite of the above fibrations is also commutative and we have,
\[\begin{tikzcd}[row sep=2em,column sep=5em]
FV_{n,2k}  \arrow[r] \arrow[d]  & G_{2k}(\mR^n)  \arrow[d] \\
\mRP^{\infty}  \arrow[r, "\Phi"] &  G_{2k}(\mR^{\infty})= BO(2k)
\end{tikzcd}\]

\noindent \underline{\bf Differentials}:
We consider three fibrations denoted by {\bf(L), (M)} and {\bf (R)} as given in the following commutative diagram. We use information of differentials in {\bf (R)} to get differentials in {\bf (M)}. Further analysis  gives us differentials in {\bf (L)}.
\[
\xymatrix{
{\bf(L)} \ar@{~>}[d] & {\bf(M)} \ar@{~>}[d] & {\bf(R)} \ar@{~>}[d]\\
V_{n,2k}\ar[d]  & O(n) \ar[d] \ar[l]\ar[r]  & O(n) \ar[d] \\
G_{2k}(\mathbb{R}^n) \ar[d]   & G_{2k}(\mathbb{R}^n) \ar@{=}[l]\ar[d] \ar[r] & EO(n) \ar[d]\\
BO(2k)= G_{2k}(\mathbb{R}^\infty) & BO(2k) \times BO(n-2k) \ar[r]^-{f} \ar[l] & BO(n)
}
\]
In the given diagram, $f$ is the classifying map corresponding to the bundle $\xi_{2k} \times \xi_{n-2k}$, where $\xi_{2k}$ is the universal bundle over $BO(2k)= G_{2k}(\mathbb{R}^\infty)$ and $\xi_{n-2k}$ is the universal bundle over $BO(n-2k)= G_{n-2k}(\mathbb{R}^\infty).$

We know that the cohomology rings of \(V_{n,2k}, ~ O(n)\) and \(BO(n)\) are given by
\begin{align*}
H^*(V_{n,2k}; \mathbb{Z}/2)  & \cong \mathbb{Z}/2[z_{n-2k},  \ldots, z_{n-1}],\\
H^\ast(O(n); \mathbb{Z}/2) & \cong \mathbb{Z}/2[y_1, y_2, \ldots, y_{n-1}],\\
H^\ast(BO(n); \mathbb{Z}/2) & \cong \mathbb{Z}/2[w_1, w_2, \ldots, w_n].
\end{align*}
From \cite{b}, we have the differentials in {\bf (R)} to be 
\[ 
\tau y_i= w_{i+1}.
\]
where $\tau$ denotes the transgression map.

Using definition of classifying maps, we have 
\[
f^\ast(w_j)= w_j(\xi_{2k} \times \xi_{n-2k})~~~\text{for}~ j= 0, 1, 2, \ldots
\]
Using \cite[Page 54, Problem 4-A]{ms}, we have a precise description of the Stiefel Whitney classes of $\xi_{2k} \times \xi_{n-2k}.$ 

Using commutativity of diagram, the differentials in {\bf (M)} are
\[
d_{i+1}(y_i)= w_{i+1}(\xi_{2k} \times \xi_{n-2k}),
\]

Let $w_i = w_i(\xi_{2k}), ~\tilde{w_i} = w_i(\xi_{n-2k})$ and let $w'=1+w'_1 +\ldots$, where $w'$ is the inverse of $w= 1+w_1+\dots+w_{2k}.$

\begin{lemma}\label{wequivalence}
{{} $\tilde{w_i}= w'_i$} for all $i \leq n-2k$.
\end{lemma}
\begin{proof}
From Proposition 11.1 of \cite{b} we have 
\[
\underset{i+j=r}\sum w_i\tilde{w_j} = 0 ~\text{for all}~ r \leq n-2k. \tag{\bf I}
\]
We also know that 
\[
(1+w_1+\dots+w_{2k})(1+w'_1 +\ldots) =1 \tag{\bf II}
\]
By comparing ({\bf I}) and ({\bf II}), we have the desired conclusion.
\end{proof}
We apply Lemma \ref{wequivalence} to the differentials in the $E_{n-2k+1}$ page,
\begin{align*}
d_{n-2k+1}(y_{n-2k}) & = w_{n-2k+1}(\xi_{2k} \times \xi_{n-2k}) \\
&= \sum\limits_{\substack{i+j=n-2k+1\\ i \leq 2k\\ j\leq n-2k}}w_i\tilde{w_j}\\& =  \sum\limits_{\substack{i+j=n-2k+1 \\ i \leq 2k\\ j\leq n-2k}}w_iw'_j \\&= -w'_{n-2k+1}.
\end{align*}
Using similar argument, it can be established that $d_r(y_{r-1}) = -w'_{r}$ for all $r \geq n-2k+1$.\\

The classes $y_1, y_2, \ldots, y_{n-2k-1}$ are not in the image of the induced map $H^*(V_{n,2k}) \to H^*(O(n))$  and for $j \geq n-2k$, the classes $z_j \in H^*(V_{n,2k})$ pull back to $y_j \in H^*(O(n)).$\\

\noindent Therefore the bundle map,
\[
\begin{tikzcd}[row sep=2em,column sep=5em]
O(n)  \arrow[r] \arrow[d] & V_{n,2k}  \arrow[d]\\
G_{2k}(\mR^n) \arrow[r] \arrow[d]  & G_{2k}(\mR^n)  \arrow[d] \\
B(O(2k) \times O(n-2k))  \arrow[r] &  G_{2k}(\mR^{\infty})
\end{tikzcd}
\] induces the differentials in {\bf (L)} to be  $\tau (z_j) = -w'_{j+1}$ for all $j \geq n-2k$. Since we are working with modulo 2, we have
\[
\tau(z_{j-1})= w_j^\prime ~~~~\text{for all}~ j> n-2k.
\]

\begin{proposition}
In the spectral sequence $E_2^{p,q} = H^p(\mRP^{\infty} ) \otimes H^q(V_{n,2k}) \Rightarrow H^{p+q}(FV_{n,2k}),$ the classes $y_j$ (for $k> n- 2k$) are transgressive and the differentials are given by $d(y_j)= \binom{k+j-1}{j} x^j.$
\end{proposition}
\begin{proof}
We have
\[\begin{tikzcd}[row sep=2em,column sep=5em]
FV_{n,2k}  \arrow[r] \arrow[d]  & G_{2k}(\mR^n)  \arrow[d] \\
\mRP^{\infty}  \arrow[r, "\Phi"] &  G_{2k}(\mR^{\infty})= BO(2k)
\end{tikzcd}\]
The map $\Phi$ classifies the bundle $k\zeta_{n,2k}.$ The {}{total} Stiefel Whitney class of this bundle is given by
\[
w(k\zeta_{n,2k})= (1+ x)^k
\]
where $x$ is first Stiefel Whitney class of the bundle $\xi_{n,2k}.$ Hence, we have
\[
 \Phi^\ast(w^\prime)= (1+ x)^{-k}= \sum_{r=0}^\infty \binom{-k}{r} x^r.\tag{\bf I}
\]
Using the definition of negative binomial coefficient modulo 2, we have
\[
\sum_{r=0}^\infty \binom{-k}{r} x^r= \sum_{r=0}^\infty \binom{r+k-1}{r} x^r. \tag{\bf II}
\]
Hence, from ({\bf I}) and ({\bf II}), it follows that
\[
\Phi^\ast(w^\prime)= \sum_{r=0}^\infty \binom{r+k-1}{r}  x^r.
\]
This allows us to conclude that the differentials are given by
\[
d(y_{j-1})= -\Phi^\ast(w_j^\prime)= \binom{k+j-1}{j} x^j.
\]
\end{proof}

\begin{theorem}[Cohomology]\label{t3.5}
Suppose $2k< n$ and  
\[
\nf{n}{2k}:= ~\displaystyle{\text{min}\bigg\{j~ \bigg|~ n-2k< j \leq n~\text{and}~  \binom{k+j-1}{j}~\text{is odd} \bigg\}}.
\]
 Then, {{} we have the following additive isomorphism given by}
$$H^\ast(FV_{n,2k};\mathbb{Z}/2)\cong \mathbb{Z}/2[x]/(x^{\nf{n}{2k}}) \otimes \Lambda(y_{n-2k}, \ldots, y_{\nf{n}{2k}- 2}, y_{\nf{n}{2k}}, \ldots, y_{n-1}),$$
where the generator $x$ is of dimension $1.$
\label{theorem of cohomology}
\end{theorem}
\begin{proof}
We have obtained the differentials of the $O(2k)$ fibration $G_{\mR}(2k, n-2k) \xrightarrow{\Psi} G_{2k}(\mR^{\infty}) $ to be \(d_r(y_{r-1}) = w_r\). Via \(\Phi\), we have also obtained the differentials of the fibration 
\(FV_{n,2k} \rightarrow \mRP^{\infty}\) to be \(d_r(y_{r-1}) = w_r(k\zeta_{n,2k})\). So the first non-zero transgression is given by  
\[ 
d_{\nf{n}{2k}}(y_{\nf{n}{2k}-1}) = x^{\nf{n}{2k}}
\]
where \(\nf{n}{2k}\) is as defined in the statement of the theorem. This implies that 
\[
 E_{\nf{n}{2k}+1} \cong \mathbb{Z}/2[x]/(x^{\nf{n}{2k}})  \otimes  \Lambda(y_{n-2k},  \cdots,  y_{\nf{n}{2k}- 2},  y_{\nf{n}{2k}},  \cdots,  y_{n-1}).
\] 
Since $x^i= 0$ for $i > \nf{n}{2k},$ there are no further non-zero differentials, 
\[
E_{\nf{n}{2k}+1} = E_{\infty}
\] and hence the theorem.
\end{proof}

An immediate consequence of the Theorem \ref{t3.5} is 
\begin{corollary}
The line bundle $\xi_{n,2k}$ is a non-trivial line bundle.
\end{corollary}

\begin{remark}
Existence of $\nf{n}{2k}$ in Theorem \ref{t3.5} is guaranteed. For, 
if there is $j\leq n-k$ such that $\displaystyle{\binom{k+j-1}{j} \neq 0},$ then we are done. If not, then use argument which is used in \cite[Example 2.5]{ss} for $l_i=1$ to conclude that finding such $j$ is equivalent to finding $j$ such that $\binom{n}{j} \neq 0$ and upper bound for this is $j= n.$
\end{remark}

We conclude this section by mentioning {}{an} universal property which is as follows:\\

\noindent{\underline{\bf Universal Property:}} The space $FV_{n,2k}$ classifies (real) line bundles $L$ for which there exists a (real) bundle $E$, of rank $(n-2k)$, such that $k(L \oplus \epsilon_\mR) \oplus E$ is a (real) trivial bundle.

\section{Parallelizability and Span vs Stable Span}\label{sparallelizability}
{{} Utilizing information from the cohomology, we gain insights into the span and stable span of $FV_{n,2k}.$ We briefly recall their definitions for completeness. For more information on span and stable span, we refer to \cite[Page 224, 229]{s}.

\begin{definition}
For a smooth manifold $M,$
\begin{align*}
\text{Span}(M) &:= \max\{r| \tau M \cong \gamma \oplus r \epsilon_{\mR}\},\\
\text{Stable Span}(M) &:= \max\{r| \tau M \oplus k \epsilon_{\mR}\cong \gamma \oplus (k+r) \epsilon_{\mR} ~\text{for some} ~k \geq 1\}
\end{align*}
where $\gamma$ is some (real) vector bundle. We say that $M$ is parallelizable (stably parallelizable respectively) if Span$(M)$= dim($M$) (Stable span$(M)$= dim($M$) respectively).  
\end{definition}
We observe that Span$(M)\leq$ Stable span $(M).$ Following is a description of the Stiefel-Whitney classes using which we compute the Span and Stable span of $FV_{n,2k}.$
}
\begin{lemma}
The total Stiefel-Whitney class of the tangent bundle of \(FV_{n,2k}\) is given by
\[
w\left( TFV_{n,2k} \right) =\sum\limits_{i=0}^{k(n-k-1)} \binom{k(n-k-1)}{i} x^i.
\]
\end{lemma}
\begin{proof}
By considering the Stiefel-Whitney class of Equation \ref{stableiso}, we have,
\[
w\left (\frac{k(k+1)}{2}(\zeta_{n,2k} \otimes  \zeta_{n,2k})  \oplus TFV_{n,2k})\right ) = w\left (nk\zeta_{n,2k} \oplus k\epsilon_{\mR} \right).
\]
Applying Whitney Product Theorem, we have,
\[
w\left (  \frac{k(k+1)}{2}(\zeta_{n,2k} \otimes  \zeta_{n,2k})\right)w\left( TFV_{n,2k} \right) = w (nk\zeta_{n,2k}).
\]
Substituting the value of $\zeta_{n,2k}$ we get,
\begin{align*}
w\left (  \frac{k(k+1)}{2} (\epsilon_\mR \oplus \xi_{n,2k}) \otimes  ( \epsilon_\mR \oplus \xi_{n,2k}))\right)w\left( TFV_{n,2k} \right) &= w (nk(\epsilon_\mR \oplus \xi_{n,2k})),\\
w\left (  \frac{k(k+1)}{2} (2\epsilon_\mR \oplus 2\xi_{n,2k})\right)w\left( TFV_{n,2k} \right) &= w (nk(\epsilon_\mR \oplus \xi_{n,2k})).
\end{align*}
Again by applying Whitney Product Theorem,\begin{align*}
\left (w (2\epsilon_\mR \oplus 2\xi_{n,2k})\right)^{  \frac{k(k+1)}{2}}w\left( TFV_{n,2k} \right) &= \left( w (\epsilon_\mR \oplus \xi_{n,2k})\right)^{nk},\\
\left (w (2\xi_{n,2k})\right)^{  \frac{k(k+1)}{2}}w\left( TFV_{n,2k} \right) &= \left( w (\xi_{n,2k})\right)^{nk},\\
\left (w (\xi_{n,2k})\right)^{ k(k+1)}w\left( TFV_{n,2k} \right) &= \left( w (\xi_{n,2k})\right)^{nk}.
\end{align*}
This gives us
\[
w\left( TFV_{n,2k} \right) =\left( w (\xi_{n,2k})\right)^{k(n-k-1)}= (1+ x )^{k(n-k-1)}= \sum\limits_{i=0}^{k(n-k-1)} \binom{k(n-k-1)}{i} x^i.
\]
    
\end{proof}
\noindent Hence, we have
\[w_i\left( TFV_{n,2k} \right) = \binom{k(n-k-1)}{i} x^i.\]

\noindent {{} Now,} if $w_i(TFV_{n,2k})\neq 0$  for some $i,$ then $TFV_{n,2k}$ is not parallelizable. So, based on previous results and a few more facts, we draw the following conclusions on parallelizability, span and stable span of \(FV_{n,2k}\).

\begin{lemma}
\begin{enumerate}
\item For $n$ even, $FV_{n,n}$ is parallelizable.
\item For $n$ odd, $FV_{n,n-1}$ is parallelizable.
\item $FV_{n,2}$ is not parallelizable.
\end{enumerate}
\end{lemma}
\begin{proof}
 We observe that $FV_{n,n}= O(n)/C_2$ and $FV_{n,n-1}= SO(n)/C_2$. So, the statements \((1)\) and \((2)\) follow from the fact that a quotient of a Lie group by a finite subgroup is parallelizable \cite[Page 236, Theorem 2.1]{s}.\\
 For \((3)\), we have $k=1$ and $\nf{n}{2}= n-1.$\\
When $n$ is odd, so is $n-2$ and 
\[ 
 \binom{n-2}{1}x= (n-2)x \neq 0.
\]
 Hence, $w_1 \neq 0$ as $1< \nf{n}{2}$ and we conclude that $FV_{n,2}$ is not parallelizable.\\
When $n$ is even, we write
$$n-2= a_\alpha 2^\alpha + a_{\alpha+1}2^{\alpha+1}+\cdots+ a_\beta 2^\beta~\text{ with}~a_\alpha \neq 0. $$
In this case, for $i= 2^\alpha,$ we get $w_i= \binom{n-2}{i} x^i \neq 0$ as $i< \nf{n}{2}.$
\end{proof}

\begin{remark}{{}
(1) Since $TFV_{n,2k}$ admits at least $k$ sections, the Euler characteristic of $FV_{n,2k}$ is zero. We also have, 
\[
w_1(TFV_{n,2k})= k(n-k-1)x.
\]
\noindent Hence, the coefficient of $w_1(FV_{n,2k})$ is non-zero iff $k$ and $n$ are odd.\\
(2) Koschorke elucidated the relationship  between span and stable span by examining the vanishing of the first few Stiefel-Whitney classes. Refer to \cite[Page 279, 281-282, Theorem 20.1, Corollory 20.9, 20.10]{k} for details.}
\end{remark}
{{} An immediate consequence of the remark is Theorem  \ref{t4.5} and Theorem \ref{t4.6}. We describe one calculation here and other calculations are done in a similar manner. Let $k \equiv 1 \pmod{4}$ and $n \equiv 0 \pmod{4}.$ Then,
$k= 4a+ 1, n= 4b$ for some $a,b \in \mathbb{N}$ and
\begin{align*}
w_1(FV_{n,2k}) &= k(n-k-1)= (4a+1)[4b-(4a+1)-1] \equiv 0\pmod{2}, \\
w_2(FV_{n,2k}) &= \frac{k(n-k-1)-[k(n-k-1)-1]}{2} \equiv 1\pmod{2},\\
dim(FV_{n,2k}) &= k(2n-2k-1)\equiv 1 \pmod{4}.
\end{align*}
Since, one of $w_i \neq 0, ~FV_{n,2k}$ is not parallelizable. Also, since $dim(FV_{n,2k}) \equiv 1 \pmod{4}$ and $w_1^2= 0,$ we conclude, using \cite[Page 281, Corollory 20.9]{k}, that Span($FV_{n,2k}$)= Stable span($FV_{n,2k}$) if $k>1.$ 
} 
By replicating the above calculation for different cases, we obtain the following table:
\begin{center}
\begin{tabular}{c|c||c|c|c||p{6cm}}
$k\pmod{4}$ & $n\pmod{4}$ & $w_1$ & $w_2$ & $k(2n-2k-1)$ & Conclusion\\
\hline\hline
0 & Anything &  0 & 0  & $0\pmod{2}$ & Span= Stable span \\
\hline
1 & 0 & 0 & 1 & $1 \pmod{4}$ & Span= Stable span when $k> 1$ and $FV_{n,2k}$ is not parallelizable\\
\cline{2-6}
 & 1 & 1 & 1 & $1\pmod{2}$ & Not parallelizable \\
\cline{2-6}
 & 2 & 0 & 0 & $1\pmod{4}$& Span= Stable span when $k> 1$ \\
 \cline{2-6}
 & $3$ & 1 & 0 & $3\pmod{4}$ & Not parallelizable \\
 \hline
2 & 0  & 0 & 1 & $0\pmod{2}$ & Not parallelizable and Span= Stable span \\
\cline{2-6}
 & 1 & 0 & 0 & $0\pmod{2}$ & Span= Stable span\\
\cline{2-6}
 & 2 & 0 & 1 & $0\pmod{2}$ & Not parallelizable and Span= Stable span \\
 \cline{2-6}
 & 3 & 0 & 0 & $0\pmod{2}$ & Span= Stable span  \\
 \hline
3 & 0 & 0 & 0 & $3\pmod{8}$ & Span= Stable span for $k>3$ \\
\cline{2-6}
 & 1 & 1 & 1 & $1\pmod{2}$ & Not parallelizable  \\
\cline{2-6}
 & 2 & 0 & 1 & $3 \pmod{4}$ & Not parallelizable \\
 \cline{2-6}
 & 3 & 1 & 0 & $1\pmod{2}$ & Not parallelizable
\end{tabular}
\end{center}
{{} Findings of this table concludes:}
\begin{theorem}[Parallelizability]\label{t4.5}
$FV_{n,2k}$ is not parallelizable when
\begin{enumerate}
\item $k \equiv 1 \pmod{2}$ and $n \equiv 1 \pmod{2},$
\item $k \equiv 1 \pmod{4}$ and $n \equiv 0 \pmod{4},$
\item $k \equiv 2 \pmod{4}$ and $n \equiv 0 \pmod{4},$
\item $k \equiv 2 \pmod{4}$ and $n \equiv 2 \pmod{4},$
\item $k \equiv 3 \pmod{4}$ and $n \equiv 2 \pmod{4}.$
\end{enumerate}
\end{theorem}

\begin{theorem}[Span and stable span]\label{t4.6}
Span ($FV_{n,2k}$)= Stable span ($FV_{n,2k}$) when
\begin{enumerate}
\item $k \equiv 0 \pmod{4}$ or  $k \equiv 2 \pmod{4},$
\item $k \equiv 1 \pmod{4}$ and $n \equiv 0 \pmod{4}$ for $k>1,$
\item $k \equiv 1 \pmod{4}$ and $n \equiv 2 \pmod{4}$ for $k>1,$
\item $k \equiv 3 \pmod{4}$ and $n \equiv 0 \pmod{4}$ for $k>3.$
\end{enumerate}
\end{theorem}

\section{Equivariant maps using index}\label{sequivariant}

{}{By examining the cohomology of $PV_{n,2k}$ and $FV_{n,2k}$, it is clear that these manifolds are not homotopically equivalent for most values of $n$ and $k.$ Futhermore, in this section, we apply the findings of the cohomological computations of $PV_{n,2k}$ and $FV_{n,2k}$ to delve into an application concerning the "Existence of equivariant maps" between Stiefel manifolds (with appropriate group actions). We also discuss a few geometric results deduced from the existence of these equivariant maps. The whole of this analysis is guided by the notion of index (as defined in \ref{dindex}). See \cite[Section 2]{ja} for a detailed discussion on the basics of index theory and \cite{p} for results on index of $X= V_{n,k}$ with antipodal action of $G= C_2=\{\pm 1\}$ on $X.$ A similar inquiry on Stiefel manifolds with respect to antipodal action is explored in \cite{r} using Steenrod squares.} 

\begin{definition}\label{dindex}
Let $G$ be a (compact Lie) group acting on a (paracompact Hausdorff) space $X.$ Then \[Index^{G}(X) = Ker(c^{\ast} : H_G^{\ast}(BG) \rightarrow H^{\ast}_G(X))\] where $c$ is the unique map from $X$ to a singleton space.
\end{definition}

{}{Recall that, a map $f: X \to Y$ between $G$-spaces $X$ and $Y$ is said to be $G$-equivariant if for all $x \in X$ and all $g \in G, ~f (gx) = g f (x)$. The existence of such an equivariant map \(f: X \to Y\) induces a (contravariant) map in the equivariant cohomology, thereby, warranting the following proposition stated in \cite[Proposition 2.3]{ja}}. 

\begin{proposition}\label{pindex}
Let H be a Lie group. If there exists a H-map  $f:X \rightarrow Y$, then for any coefficient ring $R$ 
\[
Index^{H}(Y;R) \subset Index^{H}(X;R).
\]
\end{proposition}

 Since we are interested in maps between projective Stiefel manifold and flip Stiefel manifold (both of which are quotients of the Stiefel manifold via free $C_2$ action), we use notations $H$ and $H^\prime$ to differentiate between the two $C_2$ actions involved. We denote $H'= \{ \pm 1\}$ and $H =\{ I_{2k\times 2k}, F\}$ where $F$ is as introduced in Definition \ref{d1.3}. 

$(H', V_{n,2k})$ denotes the action of $H'$ on $V_{n,2k}$ by antipodal identification and \(\np{n}{2k}\) is the corresponding value that determines its index. $(H, V_{n,2k})$ denotes the action of $H$ on $V_{n,2k}$ as defined in this paper.\\

By making use of the \(\mZ/2\mZ\)-cohomology expression of \(FV_{n,2k}\) in the definition of index, it follows that $ \nf{n}{2k}$ determines \(Index^{H}(V_{n,2k})\). So, we have the following theorem.

\begin{theorem} \label{ind}
\[Index^{H}(V_{n,2k}) = (x^{\nf{n}{2k}}), \]
where $ (x^{\nf{n}{2k}})$ is the ideal in $H^*(\mRP^{\infty}; \mZ/2) = \mZ_2[x]$ generated by $x^{\nf{n}{2k}}$  with \\
$\nf{n}{2k} = \min\left\{ j | n-2k < j \leq n,~ \binom{k+j-1}{j} \equiv 1 \pmod{2} \right\}$, as defined in Theorem \ref{theorem of cohomology}.
\end{theorem}

We begin by exploring the question of when the index is very small or very big. \\
A few direct calculations give the following results.
\begin{enumerate}
\item $\nf{2k}{2k} = \begin{cases}
1 \quad \text{when}~ k \equiv 1 \pmod{2},\\
2 \quad \text{when}~ k \equiv 2 \pmod{4}.
\end{cases}$ 
\item $\nf{2k+1}{2k}= \begin{cases}
2 \quad \text{when}~ k \equiv 1,2 \pmod{4},\\
2 \quad \text{when}~ k \equiv 1,2,5,6,9,10 \pmod{12}.
\end{cases}$
\item  $\nf{2k+2}{2k}= 3 \quad \text{when}~ k \equiv 1,5,9 \pmod{12}$.

\item $\nf{n}{2}= n-1$.
\item $\nf{n}{4}= \begin{cases}
n- 3 \quad \text{when}~  n \equiv 1\pmod{2},\\
n- 2 \quad \text{when}~  n \equiv 0 \pmod{2}.
\end{cases}$
\item $\nf{n}{6}= \begin{cases}
n- 5 \quad \text{when}~  n \equiv 1, 2 \pmod{4},\\
n- 4 \quad \text{when}~  n \equiv 0 \pmod{4},\\
n- 3 \quad \text{when}~  n \equiv 3 \pmod{4}.
\end{cases}$

\end{enumerate}
\begin{remark}\label{rindex}
 $\nf{n}{2k} \neq  n$.\\
This follows from the facts that
\begin{itemize}
\item $ n-1 = \nf{n}{2} \geq \nf{n}{4} \geq \cdots \geq \nf{n}{n}$ for $n$ even and
\item $ n-1 = \nf{n}{2} \geq \nf{n}{4} \geq \cdots \geq \nf{n}{n-1} $ for $n$ odd.\\
\end{itemize}
\end{remark}

We have a relation between (integral) binomial coefficients and binomial coefficients modulo $2$ which states that 
\begin{lemma}[Lucas' Theorem]\label{Lucas}
Let $a = \sum\limits_{i=0}^sa_i2^i$ and $b = \sum\limits_{i=0}^sb_i2^i $ be the dyadic expansions of $a$ and $b$ respectively. Then
\[
\binom{a}{b}\equiv\prod_{i=0}^{m} \binom{a_i}{b_i} \pmod{2}.  
\]
\end{lemma}

\noindent Using the above lemma, we determine conditions on $n$ and $k$ for which  $\nf{n}{2k}$ takes value $n-2k+1$  (the minimal possible value in the range it is defined) or $2^r$ for $r>0$.

\begin{prop}\label{p}
(1) Let $n-k = \sum\limits_{i=0}^{s}a_i2^i$ and $n-2k+1 = \sum\limits_{i=0}^{t}b_i2^i$. We have $\nf{n}{2k}= n-2k+1$ whenever for some $i$,  if $b_i \neq 0$ and $a_i \neq 0$. Otherwise  $\nf{n}{2k}> n-2k+1$.\\
(2) For $r>0$,  $k = 1+ \sum\limits_{i=0}^{r-1} 2^i + \sum\limits_{i=r+1}^{s} a_i2^ i$, where $a_i$ is 0 or 1,  $\nf{n}{2k}= 2^r$.
\end{prop}

\begin{proof}
(1) We wish to find conditions on $n$ and $k$ that will make $\binom{k+j-1}{j} \equiv 1 \pmod 2,$ for the minimum possible value of $j=n-2k+1$.  So, the result obtained is an immediate consequence of Lemma \ref{Lucas} which is proved by observing that the binomial expression is odd if and only if whenever $1$ appears in the dyadic expansion of $n-2k+1$, then $1$ also appears in the dyadic expansion of $n-k$ and hence the proposition.\\

(2) 
We want $\nf{n}{2k}= 2^r$. That is, we want $\binom{k+2^r-1}{2^r} \equiv 1 \pmod{2}$ and for every $j <2^r$, $\binom{k+j-1}{j} \equiv 0 \pmod{2}$.\\
 $\binom{k+2^r-1}{2^r} \equiv 1 \pmod{2}$ implies that $k-1 = \sum\limits_{\stackrel{i=0}{i\neq r}}^s  a_i2^i,$ as the presence of a non-zero $2^r$ in the dyadic expansion of $k-1$ will annihilate $2^r$ in the dyadic expansion of $k+2^r-1$, causing $\binom{k+2^r-1}{2^r}\equiv 0 \pmod{2}$.\\
Let $j = \sum\limits_{i=0}^{r-1}b_i2^i$ be the dyadic expansion of $j$. $\binom{k+j-1}{j} = \binom{(k-1)+ \sum\limits_{i=0}^{r-1}b_i2^i}{ \sum\limits_{i=0}^{r-1}b_i2^i} \equiv 0 \pmod{2}$ implies that for some $b_i \neq 0$, $\binom{a_i}{b_i} \equiv 1 \pmod{2}$. So when   $k = 1+ \sum\limits_{i=0}^{r-1} 2^i + \sum\limits_{i=r+1}^{s} a_i2^ i$, $\binom{k+j-1}{j} \equiv 0 \pmod{2}$ for every $j <2^r$.
\end{proof}

Additionally, we state a few resulting corollaries.
\begin{corollary}
(1) $\nf{n}{2k}= n-2k+1$ for $k= 1.$\\
(2) For $n = 2^r - (2^s -1)$ and $n-2k =2^s - 1$ where $ s < r-1$,  $\nf{n}{2k} \neq n-2k+1$.\\
(3) For $r > 0$, $k = 1+ \sum\limits_{\substack{i=0\\ i\neq r}}^s a_i2^i,~ \nf{n}{2k} \leq2^r$.\\
(4) For $r > 0$, $k = 1+ \sum\limits_{\substack{i=0 \\ i\neq r}}^s a_i2^i$ and $n = 2k +2^r-1, ~ \nf{n}{2k}= 2^r$.
\end{corollary}

Correlating these results, we state a theorem that discusses the possibility of having equivariant maps between Stiefel manifolds using indices calculated here and in \cite[Page 135 - 138]{p}.

\begin{theorem}
(1) There doesn't exist any $H$-equivariant map $f:V_{n,2k} \to V_{m,2l}$, for $ k = 1+ \sum\limits_{i=0}^{r-1}2^i $ and $l = 1 + \sum\limits_{ i = 0 }^{s-1}a_i2^i$, $a_i = 0 ~\text{or}~1$ and $r > s$.\\
(2) There doesn't exist any equivariant map $f: (H', SO(n)) \to ( H, V_{n,2k})$, for $ n = 2^r, ~r >0$.\\
(3) There doesn't exist any equivariant map $f: (H, SO(n)) \to (H', SO(n))$, for $ n = 2^r -1$, {}{$~r > 2$}.\\
(4) There doesn't exist any equivariant map $f:(H', S^{n-1}) \to (H,V_{n,2k})$, for $ n = 2^r -1, ~r >0$.\\
(5) If there  exists an equivariant map $f: (H, O(2k)) \to (H, O(2l))$ then  $ k \geq l$ and $k \equiv 1(2) $ and $l \equiv 0 \pmod{2}$.
\end{theorem}

\begin{proof} 
The conclusions are derived from the  values of indices under discussion, \(\nf{n}{2k}\) and \(\np{n}{2k}\), in accordance with Proposition \ref{pindex}. \\
 (1) {}{It is assumed that \(r > s.\) By Proposition \ref{p}, \(\nf{n}{2k} = 2^r > 2^s \geq \nf{m}{2l}\). That is, \(Index^{H}(V_{m,2l}) \not\subset Index^{H}(V_{n,2k}).\)\\
 (2) Here \(n=2^r\). From \cite[Corollary 1]{p}, we determine that \(\np{n}{n-1} = n\). Whereas, by Remark \ref{rindex}, \(\nf{n}{2k}\) is always less than or equal to \( n-1\). Hence, for \(n = 2^r,\) we have \(\np{n}{n-1} > \nf{n}{2k}\). So, \(Index^{H}(V_{n,2k}) \not\subset Index^{H'}(SO(n))\).\\ 
(3) Here \(n= 2^r-1\) with \(r >2\) and \(2k = n-1 \implies k = 2^{r-1}-1\). \\We calculate that \[n-k = 2^{r-1} = \sum_{i=0}^{r-1}a_i2^i, \quad \text{with} \quad  a_i = \begin{cases}
0  &\text{for}~ i < r-1,\\
1  &\text{for}~ i = r-1.
\end{cases}\] and 
\[n-2k +1 = 2 = \sum_{i=0}^{1}b_i2^{i}, \quad \text{with} \quad  b_0 = 0 \quad \text{and} \quad b_1 = 1.\]  For \(i=1,\) we have \(b_1 \neq 0\) and \(a_1 = 0.\) Hence, by Proposition \ref{p} (1), \(\nf{n}{n-1}  > 2 \).\\
Moreover, from \cite[Proposition 2]{p},  \(\np{n}{n-1} =2\). \\
Therefore,
\(Index^{H'}(SO(n)) \not\subset Index^{H}(SO(n))\).\\
(4) Here \(n \equiv 1 \pmod 2\). It follows directly from the definition of \(\np{n}{1}\) in \cite[Theorem 1]{p} that \(\np{n}{1} = n. \) By Remark \ref{rindex}, \(\nf{n}{2k} \neq n\).
This implies \(Index^{H}(V_{n,2k}) \not\subset Index^{H'}(\mS^{n-1})\).\\
(5) Based on the parity of \(k\) and \(l\), we conclude from direct calculations that \(\nf{2k}{2k} = 1\) and \(\nf{2l}{2l}\geq 2\). That is \(Index^{H}(O(2l)) \subset Index^{H}(O(2k))\) for the assumed conditions on \(k\) and \(l\). Thus, by Proposition \ref{pindex}, there is a possibility of   having an \(H\)-equivariant map \(f:O(2k) \to O(2l)\).}
\end{proof}

{{} The results for quotients of the Stiefel manifolds, and equivariant functions for the Stiefel manifolds often contain implications for problems in discrete geometry. We describe one such implication along the lines of \cite{BK21}. Let $f: S^{n-1} \to \mR^m$ be a function. Consider 
\[\hat{f}:V_{n,2k} \to \mR^{km},~ \hat{f}(v_1,\ldots, v_{2k})= (f(v_1)-f(v_2), \ldots, f(v_{2k-1}) - f(v_{2k})).\]
Clearly $\hat{f}$ is a $C_2$-equivariant map between the flip action on $V_{n,2k}$ described above and the $km$-dimensional sign representation. We would like to write down conditions when such a function is forced to have a $0$. In this case we can find $(v_1,\ldots, v_{2k})$ such that $f(v_1)=f(v_2), \ldots, ~f(v_{2k-1})=f(v_{2k})$. If $\hat{f}$ takes only non-zero values, by dividing out the norm we get a $C_2$-map to the sphere $S^{km-1}$ with the antipodal action, and thus on orbits $FV_{n,2k} \to \mR P^{km-1}$ which takes the generator of $H^\ast (\mR P^{km-1};\Z/2)$ in degree $1$ to the class $x$. It forces $x^{km}=0$. We list below when this gives a contradiction for $k\leq 3$ (see Theorem \ref{ind}). 

\begin{theorem}\label{resdiscgeo}
\begin{enumerate}
\item For any function $f:S^{n-1} \to \mR^{n-2}$, there exist orthogonal $v_1, v_2$ such that $f(v_1)=f(v_2)$.  
\item For any function $f:S^{n-1}\to \mR^{\lfloor{\frac{n}{2}}\rfloor-2}$, there exist mutually orthogonal $v_1,v_2, v_3, v_4$ such that $f(v_1)=f(v_2)$ and $f(v_3)=f(v_4)$. 
\item For any function $f: S^{n-1} \to \mR ^{\lfloor\frac{n-r_n-1}{3}\rfloor}$, there exist mutually orthogonal $v_1,v_2,v_3,v_4,v_5, v_6$ such that $f(v_1)=f(v_2)$, $f(v_3)=f(v_4)$, $f(v_5)=f(v_6)$,  where 
\[r_n=\begin{cases} 5 &\mbox{ if } n \equiv 1, 2 \pmod{4},\\
4 &\mbox{ if } n \equiv 0 \pmod{4}, \\
3 &\mbox{ if } n \equiv 3 \pmod{4}. \end{cases} \] 
\end{enumerate}
\end{theorem}
\begin{proof}
(1) In this case, $m=n-2$ and $k=1.$ \\
Hence, $x^{n-2}= 0.$ Now, $\nf{n}{2}= n-1.$ This gives us a contradiction.\\
(2) Here, $m= \lfloor\frac{n}{2}\rfloor- 2$ and $k= 2.$\\
If $n$ is odd, $x^{(n-1)-4}= 0$ and $\nf{n}{4}= n-3.$\\
If $n$ is even, $x^{n-4}= 0$ and $\nf{n}{4}= n-2.$ We get contradictions in both cases.\\
(3) Here $m= \lfloor\frac{n-r_n-1}{3}\rfloor$ and $k=3,$ we have the following information.
\[
\begin{tabular}{c|c|c|c|c}
$n$ & $r_n$ & $\lfloor\frac{n-r_n-1}{3}\rfloor$ & $3\lfloor\frac{n-r_n-1}{3}\rfloor$ & $\nf{n}{6}$ \\ \hline 
$\equiv 1 \pmod{4}$ & 5 & $\lfloor\frac{n-6}{3}\rfloor$ & $\leq n-6$ & $n-5$ \\ \hline 
$\equiv 2 \pmod{4}$ & 5 & $\lfloor\frac{n-6}{3}\rfloor$ & $\leq n-6$ & $n-5$ \\ \hline 
$\equiv 0 \pmod{4}$ & 4 & $\lfloor\frac{n-5}{3}\rfloor$ & $\leq n-5$ & $n-4$ \\ \hline 
$\equiv 3 \pmod{4}$ & 3 & $\lfloor\frac{n-4}{3}\rfloor$ & $\leq n-4$ & $n-3$
\end{tabular}
\]
Using similar argument, we arrive at contradictions in the other cases. This completes our proof.
\end{proof}

A way to obtain functions $l:S^{n-1} \to \mR$ is to fix a convex compact region in $\mR^n$ and for any direction $v$, let $l(v)$ be the distance between bounding hyperplanes perpendicular to $v$. The fact that two orthogonal directions give the same value imply that the region is bounded by a square times the orthogonal subspace. We thus get the following geometric consequences of Theorem \ref{resdiscgeo} by taking $f= l.$

\begin{theorem}
\begin{enumerate}
\item Given $n-2$ convex compact regions in $\mR^n$, there is a orthogonal decomposition $\mR^n=\mR^2 \oplus \mR^{n-2}$ such that each region is inscribed in a  square $\times \mR^{n-2}$. 
\item Given $\lfloor{\frac{n}{2}}\rfloor-2$ convex compact regions in $\mR^n$, there is a orthogonal decomposition $\mR^n=\mR^2 \oplus \mR^2 \oplus \mR^{n-4}$ such that each region is inscribed in a  square $\times$ square $\times \mR^{n-4}$. 
\item Given $\lfloor\frac{n-r_n-1}{3}\rfloor$ convex compact regions in $\mR^n$ (with $r_n$ as in Theorem \ref{resdiscgeo}), there is a orthogonal decomposition $\mR^n=\mR^2 \oplus \mR^2\oplus \mR^2 \oplus \mR^{n-6}$ such that each region is inscribed in a  square $\times$ square $\times$ square $\times \mR^{n-6}$. 
\end{enumerate}
\end{theorem}
}
\begin{ack}
We thank the referee for the useful comments and suggestions. Second author thanks Parameswaran Sankaran for discussions about the problem. Second author also thanks Prasanna Kumar for proof-reading the final draft. Third author acknowledges the University Grant Commission - Ministry of Human Resource Development, New Delhi, for the financial support granted through the CSIR-UGC fellowship. Third author also thanks the Birla Institute of Technology and Science Pilani, K K Birla Goa Campus,  where the research was conducted. 
\end{ack}

\end{document}